\documentclass[12pt,a4paper]{article}
\usepackage[T1]{fontenc}
\usepackage{amsmath,amssymb,amsthm,mathtools}
\usepackage{geometry}
\usepackage{enumitem}
\usepackage{hyperref}
\hypersetup{colorlinks=true,linkcolor=black,citecolor=black,urlcolor=black}

\theoremstyle{plain}
\newtheorem{theorem}{Theorem}
\newtheorem{lemma}{Lemma}
\theoremstyle{remark}
\newtheorem{remark}{Remark}

\title{Distance Stability for the Integral Hardy and Carleman Inequalities}
\author{Gangsong Leng\\[2mm]
School of Mathematical Sciences, East China Normal University\\
\texttt{lenggangsong@163.com}}
\date{}

\begin{document}
\maketitle

\begin{abstract}
We study stability forms of two classical integral inequalities: Hardy's integral inequality and the integral Carleman inequality, also known as the P\'olya--Knopp inequality. The sharp constants in these two inequalities are $(p')^p$ and $e$, respectively, but neither inequality has a non-trivial extremizer in its natural function space. Their formal extremal functions are $c x^{-1/p}$ and $c/x$, respectively. We first derive a deficit identity for Hardy's integral inequality. In particular, when $p=2$, the deficit is exactly a squared norm. We then prove a distance stability estimate for the integral Carleman inequality, whose remainder term directly measures a local weighted Hellinger-type distance from the family $\{c/x:c\ge0\}$. Together these results illustrate the same stability phenomenon: although the classical integral inequalities have no genuine extremizers, their deficits still measure the deviation from the corresponding families of virtual extremals.
\end{abstract}



\section{Introduction}

Hardy's integral inequality asserts that, if $p>1$, $p'=p/(p-1)$ and $f\ge0$, then
\begin{equation}\label{hardy}
\int_0^\infty \left(\frac1x\int_0^x f(t)\,dt\right)^p dx
\le (p')^p\int_0^\infty f(x)^p\,dx .
\end{equation}
The constant $(p')^p$ is sharp, but there is no non-zero extremizer in $L^p(0,\infty)$. The formal equality equation is
\[
 p'f(x)= \frac1x\int_0^x f(t)\,dt,
\]
whose non-zero solutions have the form $f(x)=c x^{-1/p}$. Since $x^{-1/p}\notin L^p(0,\infty)$, this function is only a virtual extremal.

On the other hand, the integral form of Carleman's inequality is
\begin{equation}\label{carleman}
\int_0^\infty \exp\left\{ \frac1x\int_0^x \log f(t)\,dt\right\}dx
\le e\int_0^\infty f(t)\,dt .
\end{equation}
This inequality is also called the P\'olya--Knopp inequality. It is parallel to the discrete Carleman inequality
\[
\sum_{n=1}^{\infty}(a_1a_2\cdots a_n)^{1/n}
\le e\sum_{n=1}^{\infty}a_n .
\]
Its formal extremal functions satisfy $t f(t)=c$, that is, $f(t)=c/t$. Again, since $1/t$ is not integrable on $(0,\infty)$, the integral Carleman inequality has no non-zero extremizer.

Both inequalities have a long history. Hardy gave the integral form in 1925 \cite{Hardy}, and it later became one of the fundamental results in the classical theory of inequalities developed by Hardy, Littlewood and P\'olya \cite{HardyLittlewoodPolya}. Carleman's inequality arose from Carleman's work on quasi-analytic functions \cite{Carleman}, and Knopp subsequently gave a famous proof and generalization \cite{Knopp}. For the history, proofs and extensions of Carleman's inequality and the P\'olya--Knopp inequality, see Johansson--Persson--Wedestig \cite{Johansson} and Kaijser--Persson--\"Oberg \cite{Kaijser}. In the last two decades, Hardy-type operators, P\'olya--Knopp-type operators, weighted forms, probabilistic forms and measure-theoretic forms have continued to be developed. For example, Klaassen--Wellner studied Hardy, Copson and Carleman--P\'olya--Knopp inequalities from a unified probabilistic point of view \cite{KlaassenWellner}; Arias--Rodr\'iguez-L\'opez investigated weighted Carleman-type inequalities \cite{Arias}; and Kondo--Moritoh--Tanaka gave a direct proof of a weighted P\'olya--Knopp inequality following Carleson's method \cite{Kondo}.

More closely related to the present paper are stability and remainder forms of Hardy-type inequalities. For differential Hardy inequalities, there is already a substantial literature. Cianchi--Ferone proved Hardy inequalities with non-standard remainder terms, where the remainders depend on the distance from the so-called ``virtual extremals'' \cite{CianchiFerone}. Frank--Seiringer developed nonlinear ground state representations and obtained sharp Hardy-type inequalities with remainders \cite{FrankSeiringer}. The work of Duy--Lam--Triet--Yin also emphasizes the relation between exact remainder identities and virtual extremal functions \cite{Duy}. These results show that, when a sharp inequality has no genuine extremizer, its deficit should often measure the distance from an appropriate family of virtual extremals.

However, much of the existing stability theory concerns differential Hardy inequalities, fractional Hardy inequalities, or general Hardy-type frameworks. For the original Hardy averaging inequality \eqref{hardy} and the integral Carleman/P\'olya--Knopp inequality \eqref{carleman}, the literature more often focuses on sharp constants, weighted versions, sharpened forms and generalizations. For instance, \v{C}i\v{z}me\v{s}ija--Pe\v{c}ari\'c--Persson studied strengthened Hardy and P\'olya--Knopp inequalities \cite{Cizmesija}. By contrast, for the most basic Hardy averaging inequality and the integral Carleman inequality, the viewpoint of interpreting the deficit directly as a distance from the virtual extremal families
\[
 \{c x^{-1/p}:c\ge0\},\qquad \{c/x:c\ge0\}
\]
does not appear to be a standard formulation. The present paper is written from this point of view and gives two elementary and transparent stability results.

The guiding idea is simple. Since genuine extremizers do not exist, stability should not be formulated as closeness to an actual extremizer. It should instead be formulated as closeness to the corresponding family of virtual extremals. For Hardy's inequality, the virtual extremal family is
\[
 \{c x^{-1/p}:c\ge0\},
\]
whereas for Carleman's inequality it is
\[
 \{c/x:c\ge0\}.
\]
We treat these two inequalities separately and then give a short comparison at the end.

\section{A deficit identity for Hardy's integral inequality}

Let
\[
Hf(x)=\frac1x\int_0^x f(t)\,dt .
\]
The deficit in Hardy's inequality is denoted by
\[
\Delta_H(f)=(p')^p\int_0^\infty f(x)^p\,dx
-\int_0^\infty (Hf(x))^p\,dx .
\]
The following identity gives a natural stability form of Hardy's inequality. It controls exactly the residual in the equality equation and, combined with the form of the solutions of that equation, it may be interpreted as a distance stability statement with respect to the virtual extremal family.

\begin{theorem}[Hardy deficit identity]\label{thm:hardy}
Let $p>1$ and $p'=p/(p-1)$. Let $f\ge0$ and $f\in L^p(0,\infty)$. We first assume that $f$ is compactly supported and locally bounded, in order to avoid endpoint technicalities. Put
\[
 h(x)=Hf(x).
\]
If
\[
\Phi_p(r)=r^p-1-p(r-1),\qquad r\ge0,
\]
then
\begin{equation}\label{hardy-deficit}
(p')^p\int_0^\infty f(x)^p\,dx
-\int_0^\infty h(x)^p\,dx
=
\int_0^\infty h(x)^p
\Phi_p\left(\frac{p'f(x)}{h(x)}\right)dx .
\end{equation}
At points where $h(x)=0$, the integrand on the right is understood as $0$. Indeed, since $f\ge0$, the identity $h(x)=0$ implies that $f=0$ almost everywhere on $(0,x)$. The identity extends to general non-negative $L^p$ functions by a standard truncation argument; if one of the terms is infinite, it is understood in the usual extended sense.
\end{theorem}

\begin{proof}
Let
\[
F(x)=\int_0^x f(t)\,dt,
\qquad h(x)=\frac{F(x)}{x}.
\]
Then
\[
h'(x)=\frac{f(x)-h(x)}{x}.
\]
Hence
\begin{align*}
\frac{d}{dx}\{x h(x)^p\}
&=h(x)^p+p x h(x)^{p-1}h'(x)\\
&=h(x)^p+p h(x)^{p-1}\{f(x)-h(x)\}\\
&=p f(x)h(x)^{p-1}-(p-1)h(x)^p .
\end{align*}
Since $f$ is compactly supported and locally bounded, the endpoint terms satisfy
\[
\lim_{x\to0+}x h(x)^p=0,
\qquad
\lim_{x\to\infty}x h(x)^p=0.
\]
Integrating the preceding identity over $(0,\infty)$ gives
\begin{equation}\label{basic-id}
\int_0^\infty h(x)^p\,dx
=p'\int_0^\infty f(x)h(x)^{p-1}\,dx .
\end{equation}
At points where $h(x)>0$, set
\[
r(x)=\frac{p'f(x)}{h(x)}.
\]
Then
\[
(p')^p f(x)^p=r(x)^p h(x)^p,
\qquad
p'f(x)h(x)^{p-1}=r(x)h(x)^p.
\]
By \eqref{basic-id},
\[
\int_0^\infty \{r(x)-1\}h(x)^p\,dx=0.
\]
Therefore
\begin{align*}
\Delta_H(f)
&=\int_0^\infty \{r(x)^p-1\}h(x)^p\,dx\\
&=\int_0^\infty \{r(x)^p-1-p(r(x)-1)\}h(x)^p\,dx\\
&=\int_0^\infty h(x)^p
\Phi_p\left(\frac{p'f(x)}{h(x)}\right)dx .
\end{align*}
The theorem follows.
\end{proof}

Since $r\mapsto r^p$ is convex,
\[
\Phi_p(r)=r^p-1-p(r-1)\ge0,
\]
and Theorem \ref{thm:hardy} immediately implies Hardy's integral inequality. More importantly, the right-hand side measures precisely the deviation from the equality equation
\[
p'f=Hf.
\]

When $p=2$, Theorem \ref{thm:hardy} yields a particularly simple square identity.

\begin{theorem}[Quadratic stability identity for Hardy's inequality]\label{thm:hardy2}
Let $f\ge0$ and $f\in L^2(0,\infty)$. In the above sense,
\begin{equation}\label{hardy2}
4\int_0^\infty f(x)^2\,dx
-\int_0^\infty (Hf(x))^2\,dx
=
\int_0^\infty \{2f(x)-Hf(x)\}^2\,dx .
\end{equation}
\end{theorem}

\begin{proof}
Take $p=2$ in Theorem \ref{thm:hardy}. Then $p'=2$ and
\[
\Phi_2(r)=r^2-1-2(r-1)=(r-1)^2.
\]
Thus
\[
h(x)^2\Phi_2\left(\frac{2f(x)}{h(x)}\right)
=\{2f(x)-h(x)\}^2,
\]
which gives \eqref{hardy2}.
\end{proof}

\begin{remark}
The equation $p'f=Hf$ is equivalent to
\[
p'F'(x)=\frac{F(x)}{x},
\qquad F(x)=\int_0^x f(t)\,dt .
\]
Hence $F(x)=C x^{1/p'}$ and therefore
\[
f(x)=\frac{C}{p'}x^{-1/p}.
\]
Thus a small Hardy deficit forces $f$ to be close, on the main scales, to the virtual extremal family $\{c x^{-1/p}:c\ge0\}$. A global $L^p$ distance statement must be formulated with truncations or logarithmic variables, since $x^{-1/p}\notin L^p(0,\infty)$.
\end{remark}

\begin{remark}
For $p=2$, this distance interpretation can be made more explicit. Define
\[
Q(s)=e^{-s/2}\int_0^{e^s}f(t)\,dt,
\qquad s\in\mathbb R .
\]
Then
\[
Q'(s)=\frac12 e^{s/2}\{2f(e^s)-Hf(e^s)\}.
\]
Consequently, Theorem \ref{thm:hardy2} gives
\[
4\int_{\mathbb R}|Q'(s)|^2\,ds
=
4\int_0^\infty f(x)^2\,dx-
\int_0^\infty (Hf(x))^2\,dx .
\]
The virtual extremal $f(x)=c x^{-1/2}$ corresponds to $Q(s)\equiv 2c$. Hence, on any finite logarithmic interval $I$, Poincar\'e's inequality yields
\[
\inf_{C\in\mathbb R}\int_I |Q(s)-C|^2\,ds
\le \frac{|I|^2}{\pi^2}\int_I |Q'(s)|^2\,ds
\le \frac{|I|^2}{4\pi^2}
\left(4\int_0^\infty f(x)^2\,dx-
\int_0^\infty (Hf(x))^2\,dx\right).
\]
This shows that the quadratic Hardy deficit controls the local distance, on logarithmic scales, of the primitive variable to the constants. This is a scale-invariant interpretation of the closeness of $f$ to $c x^{-1/2}$.
\end{remark}

\section{Distance stability for the integral Carleman inequality}

We now turn to the integral Carleman inequality. Let $f:(0,\infty)\to(0,\infty)$ be measurable and put
\[
G_f(x)=\exp\left\{\frac1x\int_0^x\log f(t)\,dt\right\}.
\]
The Carleman deficit is defined by
\[
\Delta_C(f)=e\int_0^\infty f(t)\,dt-
\int_0^\infty G_f(x)\,dx .
\]
In order to measure the closeness of $f$ to $c/t$, define, for each $x>0$,
\begin{equation}\label{Dx}
D_x(f)^2=
\inf_{c\ge0}\frac1x\int_0^x
 t\left(\sqrt{f(t)}-\sqrt{c/t}\right)^2dt .
\end{equation}
Equivalently,
\begin{equation}\label{Dx2}
D_x(f)^2=
\inf_{c\ge0}\frac1x\int_0^x
 \left(\sqrt{t f(t)}-\sqrt c\right)^2dt .
\end{equation}
This is a local weighted Hellinger-type distance from $f$ to the family $\{c/t:c\ge0\}$. It avoids the use of an ordinary $L^1$ or $L^2$ distance, which would be unnatural here because $1/t$ itself is not integrable on $(0,\infty)$.

\begin{theorem}[Carleman distance stability]\label{thm:carleman}
Let $f>0$, $\int_0^\infty f(t)\,dt<\infty$, and assume that $\log f$ is integrable on every finite interval. Then
\begin{equation}\label{carleman-stability}
\Delta_C(f)
\ge
 e\int_0^\infty D_x(f)^2\,\frac{dx}{x} .
\end{equation}
\end{theorem}

The proof uses the following stability form of the integral AM--GM inequality.

\begin{lemma}\label{lem:agm}
Let $u>0$ be integrable on $(0,x)$ and assume that $\log u$ is integrable. Put
\[
A=\frac1x\int_0^x u(t)\,dt,
\qquad
M=\exp\left\{\frac1x\int_0^x\log u(t)\,dt\right\}.
\]
Then
\begin{equation}\label{agm-stab}
A-M\ge \frac1x\int_0^x(\sqrt{u(t)}-\sqrt M)^2dt .
\end{equation}
\end{lemma}

\begin{proof}
Expanding the right-hand side, we obtain
\begin{align*}
\frac1x\int_0^x(\sqrt{u(t)}-\sqrt M)^2dt
&=A+M-2\sqrt M\,\frac1x\int_0^x\sqrt{u(t)}\,dt .
\end{align*}
By the AM--GM inequality,
\[
\frac1x\int_0^x\sqrt{u(t)}\,dt
\ge
\exp\left\{\frac1x\int_0^x\log\sqrt{u(t)}\,dt\right\}
=\sqrt M .
\]
Thus
\[
2\sqrt M\,\frac1x\int_0^x\sqrt{u(t)}\,dt\ge2M,
\]
and \eqref{agm-stab} follows.
\end{proof}

\begin{proof}[Proof of Theorem \ref{thm:carleman}]
Let
\[
u(t)=t f(t).
\]
For every $x>0$, define
\[
A(x)=\frac1x\int_0^x u(t)\,dt,
\qquad
M(x)=\exp\left\{\frac1x\int_0^x\log u(t)\,dt\right\}.
\]
Since
\[
\frac1x\int_0^x\log t\,dt=\log x-1,
\]
we have
\begin{equation}\label{M-G}
M(x)=\frac{x}{e}G_f(x).
\end{equation}
On the other hand, by Tonelli's theorem,
\begin{align}
\int_0^\infty \frac{A(x)}{x}\,dx
&=\int_0^\infty \frac1{x^2}\int_0^x t f(t)\,dt\,dx \notag\\
&=\int_0^\infty t f(t)\left(\int_t^\infty \frac{dx}{x^2}\right)dt \notag\\
&=\int_0^\infty f(t)\,dt . \label{tail-id}
\end{align}
By \eqref{M-G},
\[
\int_0^\infty G_f(x)\,dx
=e\int_0^\infty \frac{M(x)}{x}\,dx .
\]
Therefore
\begin{equation}\label{deficit-C}
\Delta_C(f)=e\int_0^\infty \frac{A(x)-M(x)}{x}\,dx .
\end{equation}
By Lemma \ref{lem:agm},
\[
A(x)-M(x)
\ge \frac1x\int_0^x
\left(\sqrt{t f(t)}-\sqrt{M(x)}\right)^2dt .
\]
Substituting this into \eqref{deficit-C}, we get
\begin{equation}\label{pre-distance}
\Delta_C(f)
\ge
 e\int_0^\infty \frac1{x^2}\int_0^x
\left(\sqrt{t f(t)}-\sqrt{M(x)}\right)^2dt\,dx .
\end{equation}
By the definition of $D_x(f)$, for each $x>0$,
\[
D_x(f)^2\le \frac1x\int_0^x
\left(\sqrt{t f(t)}-\sqrt{M(x)}\right)^2dt .
\]
Combining this with \eqref{pre-distance} gives
\[
\Delta_C(f)
\ge e\int_0^\infty D_x(f)^2\frac{dx}{x}.
\]
The theorem follows.
\end{proof}

\begin{remark}
The distance $D_x(f)$ has an explicit expression. Let
\[
y(t)=\sqrt{t f(t)}.
\]
Then
\[
D_x(f)^2=
\inf_{\alpha\ge0}\frac1x\int_0^x(y(t)-\alpha)^2dt .
\]
The minimizing value is attained at
\[
\alpha_x=\frac1x\int_0^x\sqrt{t f(t)}\,dt.
\]
Thus
\[
D_x(f)^2=\frac1x\int_0^x
\left(\sqrt{t f(t)}-
\frac1x\int_0^x\sqrt{s f(s)}\,ds\right)^2dt .
\]
Hence $D_x(f)^2$ is exactly the variance of $\sqrt{t f(t)}$ on $(0,x)$. It vanishes if and only if $t f(t)$ is almost everywhere constant on $(0,x)$, that is, if and only if $f(t)=c/t$ on that interval.
\end{remark}

We end with a brief comparison. The stability phenomena for the integral Hardy and Carleman inequalities have the same conceptual origin: neither inequality has a genuine non-zero extremizer, but each has a clear virtual extremal family. For Hardy's inequality, the virtual extremals are $c x^{-1/p}$; for Carleman's inequality, they are $c/x$. Accordingly, the Hardy deficit measures the deviation from the equality equation $p'f=Hf$, while the Carleman deficit measures the local deviation of $\sqrt{x f(x)}$ from constants.

In the Hardy case, stability is most directly expressed by the deficit identity
\[
\Delta_H(f)=\int_0^\infty (Hf)^p
\Phi_p\left(\frac{p'f}{Hf}\right)dx .
\]
This shows that the Hardy deficit controls exactly the residual in the equality equation $p'f=Hf$. In particular, when $p=2$, the deficit is the squared norm
\[
\int_0^\infty(2f-Hf)^2dx .
\]

In the Carleman case, stability is expressed as a local distance from the family $\{c/x:c\ge0\}$. Its clean form relies on the exact tail identity
\[
 t\int_t^\infty \frac{dx}{x^2}=1.
\]
This is also the reason why the integral Carleman inequality admits a simpler stability remainder than the discrete Carleman inequality. In the discrete setting, the corresponding tail summation usually gives only an estimate, and the resulting stability formula tends to involve less tidy Kober-type coefficients.

The two main results of this paper may therefore be summarized as follows: although the classical integral Hardy and Carleman inequalities have no genuine extremizers, their deficits still provide natural measures of deviation from the corresponding families of virtual extremals.

\end{document}